\documentclass[11pt]{amsart}
\theoremstyle{plain}
\newtheorem{thm}{Theorem}[section]
\newtheorem{theorem}[thm]{Theorem}

\newtheorem{lemma}[thm]{Lemma}

\newtheorem{proposition}[thm]{Proposition}
\theoremstyle{definition}
\newtheorem{remark}[thm]{Remark}

\newtheorem{notation}[thm]{Notation}

\newtheorem{definition}[thm]{Definition}

\numberwithin{equation}{section}

\newcommand{\sH}{{\mathcal H}}

\newcommand{\sL}{{\mathcal L}}

\newcommand{\sO}{{\mathcal O}}

\newcommand{\sW}{{\mathcal W}}
\newcommand{\sX}{{\mathcal X}}
\newcommand{\sY}{{\mathcal Y}}

\newcommand{\A}{{\mathbb A}}

\newcommand{\C}{{\mathbb C}}

\newcommand{\BP}{{\mathbb P}}

\newcommand{\Z}{{\mathbb Z}}
\newcommand{\bm}{{\mathbf{m}}}

\newcommand{\fg}{{\mathfrak g}}

\newcommand{\fp}{{\mathfrak p}}

\newcommand{\fn}{{\mathfrak n}}
\newcommand{\fa}{{\mathfrak a}}

\newcommand{\fh}{{\mathfrak h}}

\newcommand{\fc}{{\mathfrak c}}

\newcommand\p{\partial}

\newcommand\gr{\rm gr}

\def\Sym{\mathop{\rm Sym}\nolimits}

\def\Hom{\mathop{\rm Hom}\nolimits}

\title[Contact fundamental forms and adjoint varieties]{Contact fundamental forms and \\ adjoint varieties}

\author{Baohua Fu and Jun-Muk Hwang}

\thanks{Baohua Fu was supported by the National Key Research and Development Program of China (No. 2025YFA1017302). Jun-Muk Hwang was supported by the Institute for Basic Science (IBS-R032-D1). }

\begin{document} 

\begin{abstract}

We introduce contact symbol systems, a noncommutative analogue of symbol
systems for projective fundamental forms, by replacing the polynomial
algebra on a vector space by the graded dual of the universal enveloping
algebra of a Heisenberg algebra.  For a complex projective submanifold equipped
with a contact structure, we define contact fundamental forms and prove that,
at a general point, they form a contact symbol system, which gives a contact version of the classical 
result due to \'E. Cartan. Conversely, we prove that every
 contact symbol system can be realized as the contact fundamental forms of a projective variety with a
dense open Heisenberg orbit, called the Heisenberg-symmetric
variety associated to the contact symbol system.  We show that 
the closure of a  projectivized nilpotent orbit in a simple Lie algebra is Heisenberg-symmetric if and only if it is the adjoint
variety, namely, the projectivization of the minimal nilpotent orbit.  For adjoint varieties of non-symplectic simple
Lie algebras, we prove the contact analogue of the Landsberg--Manivel strict
prolongation property by using Yamaguchi's prolongation theory.

\end{abstract}

\maketitle


\section{Introduction}
We work over the complex numbers. All varieties are assumed to be irreducible. For a vector space $V$, its projectivization $\BP V$ is the set of one-dimensional subspaces of $V$. 

\medskip To motivate our results, let us recall a few basic facts on fundamental forms of projective varieties. It is convenient to start with the following definition.

\begin{definition}\label{d.symbol}
Let $W$ be a vector space and let $\Sym W^* = \oplus_{k \in \Z} \Sym^k W^*$ be the graded ring of polynomial functions on $W$, regarded as an algebraic variety.  Here, we use the convention $\Sym^k W^* =0$ if $k<0$.
\begin{itemize} \item[(i)] For each $w \in W$, let $\p_w$ be the vector field on $W$ corresponding to $w$ under the canonical trivialization $T W = W \times W$ and regard it as an operator $\p_w \in \Hom(\Sym W^*, \Sym W^*) $ sending $\Sym^{k} W^*$ to $\Sym^{k-1} W^*$ for any $k$.  \item[(ii)]  A finite-dimensional graded vector subspace $$S = \oplus_{k \in \Z} S^k \subset \Sym W^*, \ S^k \subset \Sym^k W^*,$$ is called a {\em symbol system  on $W$} if $S^0 = \C, \ S^1 = W^*$ and  $\p_w S^k \subset S^{k-1}$ for any $k\geq 2$ and $ w \in W$. 
 \end{itemize} \end{definition}

 The most important examples of symbol systems are systems of fundamental forms of a submanifold in  projective space (see  \cite[Definition 2.4]{FH}), from the following result of \'E. Cartan (\cite[Theorem 3.3]{FH}, \cite[p.63]{LM03}).
 
 \begin{theorem}\label{t.Cartan}
 Let $M \subset \BP^N$ be a locally closed submanifold in projective space (for example, the smooth locus of a projective variety). Then for a general  point $z \in M$, the system of fundamental forms $S_{z}$ at $z$ is a symbol system on the tangent space $T_z M$.  \end{theorem}
 
 The converse of Theorem \ref{t.Cartan} is the following result from \cite[Section 3]{FH}.

  \begin{theorem}\label{t.FH}
For any symbol system $S \subset \Sym W^*,$ there is a projective variety $Z^S \subset \BP S^* = \BP^N, N+1 = \dim S$ with a Zariski-open subset $M$ in the smooth locus of $Z^S$ with the following properties.
\begin{itemize}
\item[(i)] For any $z \in M$, the system of fundamental forms $S_z \subset \Sym T^*_z M$ at $z$ is isomorphic to $S \subset \Sym W^*$ via a linear isomorphism $T_z M \cong W$.
\item[(ii)] For any $z \in M$, there exists a subgroup $\C^* \subset {\rm PGL}(\C^{N+1})$,  whose  action on $\BP^N$ preserves $M$ with $z$ as an isolated fixed point, such that the induced $\C^*$-action on $T_z M$ is by scalar multiplications.
    \end{itemize}
    Moreover, any nondegenerate projective variety $Z \subset \BP^N$ with a Zariski-open subset $M$ in its smooth locus, satisfying  two conditions (i) and (ii) is isomorphic to $Z^S \subset \BP S^*$ by a projective linear isomorphism $\BP^N \cong \BP S^*$. \end{theorem}
    
The projective variety $Z^S \subset \BP S^*$ in Theorem \ref{t.FH} is called  {\em the Euler-symmetric variety}  associated to the symbol system $S$. The best-known examples of Euler-symmetric varieties are  minimal embeddings of irreducible Hermitian symmetric spaces. Landsberg and Manivel discovered in \cite[Theorem 3.1]{LM03} that the system of fundamental forms of the minimal embedding of an irreducible Hermitian symmetric space has the following remarkable property, which they called the strict prolongation property. 
    
    \begin{theorem}\label{t.LM}
    Let $z \in Z \subset \BP^N$ be a point in the minimal embedding of an irreducible Hermitian symmetric space. Then for the system of fundamental forms $S_z = \oplus_{k \geq 0} S_z^k, S_z^k \subset \Sym^k W^*, W= T_z Z$,   we have 
    $$ S_z^{k+1} = \{ F \in \Sym^{k+1} W^* \mid \p_w F \in S_z^k \mbox{ for all } w \in W\}$$ for all $k \geq 2$. \end{theorem} 
 
\medskip
The goal of this paper is to introduce a noncommutative analogue of Definition \ref{d.symbol}, replacing the vector space $W$ by the Heisenberg algebra and to formulate noncommutative analogues of Theorems \ref{t.Cartan}, \ref{t.FH} and \ref{t.LM} when there is a  contact structure on the submanifold of projective space. 

The correct analogue of Definition \ref{d.symbol} turns out to be the following.

\begin{definition}\label{d.Csymbol}
Let $\fh = \fh_1 + \fh_2$ be a graded Heisenberg algebra.  Let $U(\fh) = \oplus_{k \geq 0} U_k(\fh)$ be its  universal enveloping algebra with the grading induced by the grading of $\fh$ and let $U(\fh)^* = \oplus_{k \geq 0} U_k(\fh)^*$ be its graded dual.
We can identify $U(\fh)^*$ with the space of polynomial functions on the Heisenberg group $\sH$ of $\fh$ by regarding $U(\fh)$ as the space of left-invariant differential operators on $\sH$ (see Remark \ref{r.jet}). Let $\iota:\fh_2\hookrightarrow U_2(\fh)$ be the natural inclusion and let $\rho_2:=\iota^*:U_2(\fh)^*\to\fh_2^*$ be the restriction map.  
 A finite-dimensional graded subspace $$S = \oplus_{k \geq 0} S^k \subset U(\fh)^*, S^k \subset U_k(\fh)^*,$$ is called a {\em contact symbol system} on $\fh$,
if $S^0 = \C, S^1 = \fh_1^*$, $\rho_2(S^2)=\fh_2^*$, and $S$ is preserved under the operation of right-invariant vector fields on $\sH$.  The {\em rank} of a contact symbol system $S$ is ${\rm rk}(S):=\max\{k\mid S^k\ne0\}.$
  \end{definition}

 The operator $\p_w$ in Definition \ref{d.symbol} is exactly right-invariant vector fields on the algebraic group $W$ of the commutative Lie algebra $W$. In this sense, Definition \ref{d.Csymbol} is a noncommutative version of Definition \ref{d.symbol}.


When $z \in M \subset \BP^N$ is a point on a submanifold equipped with a contact structure $H \subset TM$, we introduce in Definition \ref{d.CF} the {\em system of contact fundamental forms } at $z$, which is a finite-dimensional  graded subspace $S_z \subset U(\fn_z)^*$ for a graded Heisenberg algebra $\fn_z = \fn_{z,1} + \fn_{z,2}$, determined by the contact structure $H$. 
Then we prove the following analogue of Theorem \ref{t.Cartan}.

\begin{theorem}\label{t.cCartan}
 Let $M \subset \BP^N$ be a locally closed submanifold equipped with a contact structure $H \subset TM$. Then for a general  point $z \in M$, the system of contact fundamental forms $S_{z}$ at $z$ is a contact symbol system on $\fn_z$.   \end{theorem}
 
The converse of Theorem \ref{t.cCartan} is the following analogue of Theorem \ref{t.FH}.

  \begin{theorem}\label{t.cFH}
For a contact symbol system $S \subset U(\fh)^*$ on a graded Heisenberg algebra $\fh$,  there is a projective variety $Z^S \subset \BP^N = \BP S^*, N = \dim S -1$ with a Zariski-open subset $M$ in the smooth locus of $Z^S$ equipped with a contact structure $H \subset TM$, which satisfies the following properties.
\begin{itemize}
\item[(i)] For any $z \in M$, the system of contact fundamental forms $S_z \subset U(\fn_z)^*$ at $z$ is isomorphic to $S \subset U(\fh)^*$ via a graded Lie algebra isomorphism $\fn_z \cong \fh$.
\item[(ii)] For any $z \in M$, there exists a subgroup $\C^* \subset {\rm PGL}(S^*) = {\rm PGL}(N+1)$, whose action on $\BP^N$ preserves $M$ and the contact structure $H \subset TM$, and has $z$ as an isolated fixed point such that the induced $\C^*$-action on $H_z$ is by scalar multiplications.
    \end{itemize}
   Moreover, any nondegenerate projective variety $Z \subset \BP^N$ with a Zariski-open subset $M$ in its smooth locus equipped with a contact structure $H \subset TM$ satisfying  two conditions (i) and (ii) is isomorphic to $Z^S \subset \BP S^*$ by a projective linear isomorphism $\BP^{N} \cong \BP S^*$. \end{theorem}

We call the projective variety $Z^S \subset \BP S^*$ in Theorem \ref{t.cFH}  {\em the Heisenberg-symmetric variety}  associated to the contact symbol system $S$. The name reflects the fact that $Z^S$ is an equivariant compactification of a Heisenberg group.

A well-known example of  a submanifold in $\BP^N$ equipped with a contact structure is the projectivization $\BP \sO \subset \BP \fg$ of a nonzero nilpotent coadjoint orbit $\sO \subset \fg^*$ in a simple Lie algebra $\fg$. We show in Theorem \ref{t.nilpotent} that the closure   $\BP \overline{\sO} $  in $\BP \fg^*$ of $\BP \sO$ is a Heisenberg-symmetric variety if and only if $\sO$ is a minimal nilpotent orbit. In the latter case, the variety $\BP \sO$ is closed and called the {\em adjoint variety} of the simple Lie algebra $\fg$. Then we prove the following analog of Theorem \ref{t.LM}.


\begin{theorem}\label{t.cLM} Let $\fg$ be a simple Lie algebra, not of symplectic type. 
Let $z \in Z \subset \BP \fg^*$ be a point on the adjoint variety $Z$ of $\fg$.  Then for the system of contact fundamental forms $S_z = \oplus_{k \geq 0} S_z^k$, $S_z^k \subset U_k(\fn_z)^*$,   we have 
    $$ S_z^{k+1}  =  \left\{ F \in U_{k+1}(\fn_z)^* \mid  \begin{array}{l}  \sX_v (F) \in S_z^k \mbox{ for all } v \in \fn_{z,1} \mbox{ and } \\  \sX_v (F) \in S_z^{k-1} \mbox{ for all } v\in \fn_{z,2} \end{array} \right\} $$ for all $k \geq 2$, where $\sX_v$ is the right-invariant vector field on the Heisenberg group of $\fn_z$ whose value at the origin is $v \in \fn_z$.  \end{theorem} 
    
The equality in Theorem \ref{t.cLM} fails if $\fg$ is a symplectic Lie algebra (see Remark \ref{r.symplectic-exception}), in which case the adjoint variety is the second Veronese embedding of projective space.
     
     It is easy to see that the strict prolongation property of the fundamental forms, namely, the equality in Theorem \ref{t.LM}, does not hold for the adjoint variety. It is only when we consider the contact fundamental forms, that we have the analogous result of Theorem \ref{t.cLM}.

\medskip
Although Theorem \ref{t.cCartan} is a  natural noncommutative analogue of Theorem \ref{t.Cartan},  its proof, given in Section \ref{s.cFF},  is quite different from Cartan's proof of Theorem \ref{t.Cartan}, which used moving frame computations.  Se-ashi has given an alternative proof  of Theorem \ref{t.Cartan}  in \cite[Proposition 1.5.1, Corollary 1.5.2]{Se} using Spencer's differential operator $D$ on jet bundles. We prove Theorem \ref{t.cCartan} by generalizing Se-ashi's argument. For this, we develop in Section \ref{s.Spencer}, an analogue of Spencer's operator in the context of contact structures. 

The proof of Theorem \ref{t.cFH}, given in Section \ref{s.Hsymmetric}, is easy. The essential point is that the condition for the contact symbol system $S$ is given by  the right-invariant vector fields on a Heisenberg group, which makes $S$ preserved under left-translations. 

The proof of Theorem \ref{t.cLM}, presented in Section \ref{s.adjoint}, is very different from that of Theorem \ref{t.LM} in \cite{LM03}. Landsberg--Manivel's proof uses Littelmann's results on the universal enveloping algebra of the unipotent radical of the Borel subalgebra of a simple Lie algebra. Our proof of Theorem \ref{t.cLM} uses Yamaguchi's result (see Theorem \ref{t.Yamaguchi}) from \cite{Ya}, identifying the simple Lie algebra with the universal prolongation of $(\fg_-,\fg_0)$. This result of Yamaguchi's does not hold when $\fg$ is a symplectic Lie algebra.

\section{Preliminaries on Heisenberg groups and Heisenberg algebras}\label{s.prelim}

    \begin{notation}\label{n.basis}
    Let $\fh = \fh_1 + \fh_2$ be a graded Heisenberg algebra: its center is $\fh_2$ of dimension $1$ and  the Lie bracket determines a symplectic form $ \wedge^2 \fh_1 \to \fh_2$ on $\fh_1$.
  Let $\sH$ be the Heisenberg group of the Lie algebra $\fh$. We denote by  $0 \in \sH$ the identity element.  
The exponential map $ \exp: \fh \to \sH$ is a biregular morphism of algebraic varieties and its inverse is denoted by $\log : \sH \to \fh$. Fix a nonzero element $T \in \fh_2$ and a basis $$\{X_1, \ldots, X_m, Y_1, \ldots, Y_m\}$$ of $ \fh_1, \dim \fh_1 = 2m,$ satisfying $[X_a, Y_b] = \delta_{ab} T$ for all $1 \leq a, b \leq m$. 
    \end{notation}

    \begin{lemma}\label{l.coordi} In Notation \ref{n.basis}, let $x_a, y_a, t$ be the regular functions on $\sH$ determined by $$\log h \ = \ \sum_{a=1}^m (x_a(h) X_a + y_a(h) Y_a) + t(h) T 
$$ for all $h \in \sH$, defining a global coordinate system $(x_1, \ldots, x_m, y_1, \ldots, y_m, t)$  on $\sH$. Then the following holds.  \begin{itemize} \item[(i)] The group multiplication is expressed as \begin{eqnarray*} \lefteqn{(x_1, \ldots, x_m, y_1, \ldots, y_m, t) \cdot (z_1, \ldots, z_m, w_1, \ldots, w_m, s) =} \\ & & (x_1 + z_1, \ldots, x_m + z_m, y_1 + w_1, \ldots, y_m + w_m, t + s + \frac{1}{2}\sum_{a=1}^m (x_a w_a - y_a z_a)). \end{eqnarray*}
\item[(ii)] The elements  $X_a, Y_a, T \in \fh$ regarded as left-invariant vector fields on $\sH$ are expressed as
    $$X_a = \frac{\p}{\p x_a} - \frac{1}{2} y_a \frac{\p}{\p t}, \ Y_a = \frac{\p}{\p y_a} + \frac{1}{2} x_a \frac{\p}{\p t}, \ T = \frac{\p}{\p t}.$$
\item[(iii)] The space of  right-invariant vector fields   on $\sH$ are spanned by 
    $$\sX_a := \frac{\p}{\p x_a} + \frac{1}{2} y_a \frac{\p}{\p t}, \ \sY_a := \frac{\p}{\p y_a} - \frac{1}{2} x_a \frac{\p}{\p t}, \ T = \frac{\p}{\p t}.$$
\item[(iv)] The natural left-invariant contact structure $H \subset T\sH$ on $\sH$ determined by $\fh_1$ is given by the left-invariant 1-form $$ \theta = \sum_{a=1}^m(x_a {\rm d} y_a - y_a {\rm d} x_a) \ - \ 2 {\rm d} t.$$
    \end{itemize} \end{lemma}

\begin{proof} (i) is by the Campbell-Hausdorff formula. The left-invariant and right-invariant vector fields in (ii) and (iii) are obtained by differentiating the maps $g\mapsto h g$ and $g\mapsto g h$ at the identity, using the group law in (i).  Substituting these vector fields into the displayed one-form gives the identity in (iv).\end{proof}

%

Recall the following from  \cite[Theorem 7.1, Chapter I]{Ko} (also \cite[Proposition 1.1]{Be}). 

\begin{theorem}\label{t.Ko}
Let $M$ be a complex manifold with a contact structure $ D \subset TM$. 
Denote by $\fc \subset H^0(M, TM)$ the Lie algebra of contactomorphic vector fields on $M$, namely, holomorphic vector fields on $M$ whose local flows preserve $D \subset TM$. Let $L := TM/D$ be the line bundle with the quotient map $TM \to L$. Then the induced homomorphism $H^0(M, TM) \to H^0(M, L)$ sends $\fc$ isomorphically to $H^0(M, L)$.\end{theorem}

When $M$ is the Heisenberg group $\sH$, we can state it more explicitly as follows (see, for example, \cite[Theorem 2.2]{Hw}).

\begin{proposition}\label{p.Heis}
Let us use the coordinates and the 1-form $\theta$ on the Heisenberg group $\sH$ in Lemma \ref{l.coordi}.   
For each holomorphic function $F$ on $\sH$,  
  define a holomorphic vector field $\overrightarrow{F}$ on $\sH$ by \begin{eqnarray*}
\overrightarrow{F} & := &  - \frac{1}{2} \sum_{a=1}^m (\frac{\p F}{\p y_a} + \frac{x_a}{2} \frac{\p F}{\p t}) \frac{\p}{\p x_a} + \frac{1}{2} \sum_{a=1}^m (\frac{\p F}{\p x_a} - \frac{y_a}{2} \frac{\p F}{\p t}) \frac{\p}{\p y_a} \\ & & +  \left(\frac{1}{4} \sum_{a=1}^m  (x_a \frac{\p F}{\p x_a} + y_a \frac{\p F}{\p y_a}) - \frac{F}{2} \right) \frac{\p}{\p t}. \end{eqnarray*} 
 Then $\overrightarrow{F}$ is the unique vector field satisfying  $\theta (\overrightarrow{F}) = F$  and  $$  {\rm Lie}_{\overrightarrow{F}} \theta = -\frac{1}{2} \frac{\p F}{\p t} \theta.$$
\end{proposition}

\begin{proof}
Substitution of the displayed formula for $\overrightarrow F$ into $\theta$
gives $\theta(\overrightarrow F)=F$.  A direct application of Cartan's formula
${\rm Lie}_{\overrightarrow F}\theta=\iota_{\overrightarrow F}{\rm d}\theta
+{\rm d}F$ gives the second identity.  These two identities also prove
uniqueness: the difference of two such vector fields is a section of the
contact distribution annihilated by ${\rm d}\theta$ on that distribution,
and ${\rm d}\theta|_H$ is nondegenerate.
\end{proof}

\begin{lemma}\label{l.bracket}
In the notation of Lemma \ref{l.coordi} and Proposition \ref{p.Heis},
we have for each $1 \leq a \leq m$, \begin{equation}\label{eq.Xa} \sX_a = -2 \overrightarrow{y_a}, \ \sY_a = 2 \overrightarrow{x_a}, \ T = - \overrightarrow{2} \end{equation}  and for any holomorphic function $F$ on $\sH$,
\begin{equation}\label{eq.F} [\sX_a, \overrightarrow{F}] = \overrightarrow{\sX_a(F)}, \ [\sY_a, \overrightarrow{F}] = \overrightarrow{\sY_a(F)}, \ [T, \overrightarrow{F}] = \overrightarrow{T(F)}. \end{equation}  \end{lemma}

\begin{proof}
The equalities (\ref{eq.Xa}) are immediate from Proposition \ref{p.Heis}. The equalities (\ref{eq.F}) are also easy.
Let us check just the first equality. Note that the bracket $[\sX_a, \overrightarrow{F}]$ is contactomorphic, because both $\sX_a$ and $ \overrightarrow{F}$ are contactomorphic. Thus  it is enough to check $\theta([\sX_a, \overrightarrow{F}]) = \sX_a(F).$
This follows from  $$ {\rm d} \theta (\sX_a, \overrightarrow{F})= \sX_a (\theta(\overrightarrow{F})) - \overrightarrow{F}(\theta(\sX_a)) - \theta ([\sX_a, \overrightarrow{F}]) $$ and the vanishing of $$
{\rm d} \theta(\sX_a, \cdot) + {\rm d} (\theta(\sX_a))  = {\rm Lie}_{\sX_a} \theta, $$ from Proposition \ref{p.Heis}. \end{proof}

\begin{notation}\label{n.sum}  Let $U (\fh)$ be the universal enveloping algebra of the graded Heisenberg algebra $\fh = \fh_1 + \fh_2$.
Using Notation \ref{n.basis}, for nonnegative integers
$i_1,\ldots,i_m,j_1,\ldots,j_m$, denote by
$$
\{\prod_{a=1}^m X_a^{i_a}Y_a^{j_a}\}\in U(\fh)
$$
the sum of all monomials in $U(\fh)$ where $X_a$ appears $i_a$ times and
$Y_a$ appears $j_a$ times for all $1\leq a\leq m$.  For example,
$$
\{X_1 X_2 Y_3\} = X_1 X_2 Y_3 + X_2 X_1 Y_3 + X_1 Y_3 X_2
+ X_2 Y_3 X_1 + Y_3 X_1 X_2 + Y_3 X_2 X_1,
$$
and
$$
\{ X_1 Y_2^3\} = X_1 Y_2^3 + Y_2 X_1 Y_2^2
+ Y_2^2 X_1 Y_2 + Y_2^3 X_1.
$$
Any braced expression with a negative exponent is interpreted as zero.
\end{notation}

\begin{lemma}\label{l.sym-left}
For  $v_1,\ldots,v_N\in\fh_1$, put
$$
{\rm Sym}(v_1\cdots v_N):=\frac1{N!}\sum_{\sigma\in\mathfrak S_N}
v_{\sigma(1)}\cdots v_{\sigma(N)}.
$$
Let us use the convention ${\rm Sym}(v_1 \cdots v_N)=1$ if $N=0$ and ${\rm Sym}(v_1 \cdots v_N)=0$ if $N<0$.  Then for any $x \in \fh_1$, 
\begin{equation}\label{eq.sym-left}
x\,{\rm Sym}(v_1\cdots v_N)
={\rm Sym}(xv_1\cdots v_N)
+\frac12\sum_{r=1}^N [x,v_r]\,
{\rm Sym}(v_1\cdots\widehat v_r\cdots v_N),
\end{equation}
where $v_1\cdots\widehat v_r\cdots v_N$ means the collection of vectors with $v_r $ removed. 
\end{lemma}

\begin{proof}
We prove the identity by induction on $N$.  The case $N=0$ is immediate, and
the case $N=1$ is the equality
$$
xv_1=\frac12(xv_1+v_1x)+\frac12[x,v_1].
$$
Assume $N\geq2$ and that the assertion is known for $N-1$ elements of $\fh_1$.  For $1 \leq r \leq N$, write
$$ V_{\widehat r} :=  v_1\cdots\widehat v_r\cdots v_N $$ and  for $1 \leq r \neq s \leq N$, write 
$$V_{\widehat r \widehat s} = V_{\widehat s \widehat r}  :=  v_1\cdots\widehat v_r\cdots \widehat v_s \cdots  v_N, $$ the collection of vectors with $v_r$ and $v_s$ removed. 
The first-letter decomposition gives
$$
{\rm Sym}(v_1\cdots v_N)
=\frac1N\sum_{r=1}^N v_r\,{\rm Sym}(V_{\widehat r}).
$$
Since the commutators $[x,v_r]$ are central elements in the Heisenberg algebra, applying
the induction hypothesis to each $V_{\widehat r}$ gives
\begin{eqnarray*}
x\,{\rm Sym}(v_1\cdots v_N)
&=&\frac1N\sum_r (v_rx+[x,v_r]){\rm Sym}(V_{\widehat r})\\
&=&\frac1N\sum_r v_r\,{\rm Sym}(xV_{\widehat r})
+\frac1{2N}\sum_r\sum_{s\ne r}[x,v_s]\,v_r\,
{\rm Sym}(V_{\widehat r\widehat s})\\
&&+\frac1N\sum_r[x,v_r]\,{\rm Sym}(V_{\widehat r}).
\end{eqnarray*}
For a fixed $s$, another first-letter decomposition gives
$$
\sum_{r\ne s}v_r\,{\rm Sym}(V_{\widehat r\widehat s})
=(N-1){\rm Sym}(V_{\widehat s}).
$$
Thus $$\sum_r\sum_{s\ne r}[x,v_s]\,v_r\,
{\rm Sym}(V_{\widehat r\widehat s}) = \sum_s \sum_{r \neq s} [x, v_s] v_r {\rm Sym}(V_{\widehat r \widehat s}) = (N-1) \sum_s [x, v_s]  {\rm Sym}(V_{\widehat s}).$$ 
Setting 
$$
C:=\sum_r[x,v_r]{\rm Sym}(V_{\widehat r})
\quad\mbox{and}\quad
D:=\sum_r v_r\,{\rm Sym}(xV_{\widehat r}),
$$
the above computation says
$$
x\,{\rm Sym}(v_1\cdots v_N)=\frac1ND+\frac{N+1}{2N}C.
$$
On the other hand, the first-letter decomposition of
${\rm Sym}(xv_1\cdots v_N)$ is
$$
{\rm Sym}(xv_1\cdots v_N)
=\frac1{N+1}x\,{\rm Sym}(v_1\cdots v_N)+\frac1{N+1}D.
$$
Eliminating $D$ from the last two displayed equalities gives
\eqref{eq.sym-left}.  This completes the induction.
\end{proof}

\begin{lemma}\label{l.Moyal}
In Notation \ref{n.sum}, the following identities hold in $U(\fh)$ for each $1 \leq b \leq m$. 
\begin{eqnarray*} X_b\{\prod_{a=1}^m X_a^{i_a} Y_a^{j_a}\} &=& \frac{i_b +1}{\sum_{s=1}^m(i_s + j_s) +1}\{X_b^{i_b+1} Y_b^{j_b} \prod_{a \neq b} X_a^{i_a} Y_a^{j_a}\} \\ & & + \frac{\sum_{s=1}^m(i_s + j_s)}{2} \{X_b^{i_b} Y_b^{j_b-1} \prod_{a \neq b} X_a^{i_a} Y_a^{j_a}\}T, \\
 Y_b\{\prod_{a=1}^m X_a^{i_a} Y_a^{j_a}\} &=& \frac{j_b +1}{\sum_{s=1}^m(i_s + j_s) +1}\{X_b^{i_b} Y_b^{j_b+1} \prod_{a \neq b} X_a^{i_a} Y_a^{j_a}\} \\ & &  - \frac{\sum_{s=1}^m(i_s + j_s)}{2} \{X_b^{i_b-1} Y_b^{j_b} \prod_{a \neq b} X_a^{i_a} Y_a^{j_a}\} T. \end{eqnarray*}
\end{lemma}

\begin{proof}
Set $N=\sum_a(i_a+j_a)$ and write $I!J!:=\prod_a (i_a!) \prod_a (j_a!)$.  Then $$
N! \ {\rm Sym}\left(\prod_a X_a^{i_a}Y_a^{j_a}\right) =\frac{I!J!}{N!} \{\prod_a X_a^{i_a}Y_a^{j_a}\}.
$$
In fact, the left hand side is the sum of all $N!$ permutations of $X_1^{i_1} \cdots X_m^{i_m} Y_1^{j_1} \cdots Y_m^{j_m}$, while each 
 distinct word  in $\{\prod_a X_a^{i_a}Y_a^{j_a}\}$  occurs exactly
$\prod_a (i_a!) \prod_a (j_a!)=I!J!$ times in this sum, because permutations among $i_a$
 copies of  $X_a$ and $j_a$ copies of $Y_a$ for $1 \leq a \leq m$ do not change the word. 
In particular, this gives
$$
{\rm Sym}\left(X_b\prod_a X_a^{i_a}Y_a^{j_a}\right) = \frac{I!(i_b+1)J!}{(N+1)!} \{X_b^{i_b+1}Y_b^{j_b}\prod_{a\ne b}X_a^{i_a}Y_a^{j_a}\}.
$$
Regard the left hand side of the above equation as  the first term in \eqref{eq.sym-left} with $x=X_b$ and $v_1 \cdots v_N = \prod_a X_a^{i_a} Y_a^{j_a}$.  Then in the corresponding second term in \eqref{eq.sym-left},  only the $j_b$ copies of $Y_b$ have nonzero bracket $[X_b,Y_b]=T$ with $X_b$, yielding
$$
\frac12\,j_b\,
{\rm Sym}\left(X_b^{i_b}Y_b^{j_b-1}
\prod_{a\ne b}X_a^{i_a}Y_a^{j_a}\right)T
=\frac{I!J!}{2(N-1)!}
\{X_b^{i_b}Y_b^{j_b-1}\prod_{a\ne b}X_a^{i_a}Y_a^{j_a}\}T.
$$
Summing-up and cleaning the common factors, we obtain the first identity of Lemma \ref{l.Moyal}.

The second identity can be proved by the same argument, regarding the left hand side of 
$$
{\rm Sym}\left(Y_b\prod_a X_a^{i_a}Y_a^{j_a}\right) = \frac{I!J!(j_b+1)}{(N+1)!} \{X_b^{i_b}Y_b^{j_b+1}\prod_{a\ne b}X_a^{i_a}Y_a^{j_a}\}
$$
 as  the first term in \eqref{eq.sym-left} with $x=Y_b$ and using
$[Y_b,X_b]=-T$.  
\end{proof}

\begin{remark} Lemma \ref{l.Moyal} is a special case of the Moyal product formula in the theory of quantization  (see, for example, \cite[page 235]{Yo}).  \end{remark}

\section{Spencer's differential operator for contact jet bundles}\label{s.Spencer}

 Cartan's original proof of Theorem \ref{t.Cartan} seems not suitable to adapt to prove Theorem \ref{t.cCartan}. Instead, we take the approach of Se-ashi (\cite[Corollary 1.5.2]{Se}) using Spencer's differential operator $D$ on jet bundles. 
In this section, we introduce the notion of jet bundles for contact manifolds and define an analogue of Spencer's differential operator in this setting.

\begin{definition}\label{d.contact}
Let $M$ be a complex manifold  with a contact structure $H \subset TM$. 
\begin{itemize} \item[(i)]
For each $z \in M$, define $\fn_{z, 1} := H_z$ and $\fn_{z, 2} := T_z M/H_z$. Then $\fn_z = \fn_{z,1} + \fn_{z, 2}$ has a graded Heisenberg algebra structure given by Lie brackets of vector fields. We call it the {\em symbol algebra} of the contact structure $H \subset TM$ at $z \in M$.  Its universal enveloping algebra $U(\fn_z)$ has a grading coming from the grading $\fn_z = \fn_{z,1} + \fn_{z, 2}$ (this grading is not compatible with the natural filtration of the universal enveloping algebra), which we denote by $$U(\fn_z) \  = \ \oplus_{k \geq 0} U_k(\fn_z).$$  As $z$ varies, it defines a vector bundle $U(\fn)$ on $M$ with the fiber $U(\fn)_z = U(\fn_z)$  and a vector subbundle $U_k(\fn)$ for each $k \geq 0$ with the fiber  $U_k(\fn)_z = U_k(\fn_z)$.
\item[(ii)] For each $z \in M$,  define the {\em weighted order} ${\rm ord}_z(\varphi)$ of a local holomorphic function $\varphi \in \sO_{M, z}$ to be  zero.
The {\em weighted order} ${\rm ord}_z (\xi)$ of a nonzero local vector field $\xi$ is defined to be 1 if it is a local section  of $H$ and 2 otherwise.  A local monomial  differential operator  on $M$ of the form $\xi_1 \cdots \xi_r$ has weighted order $${\rm ord}_z (\xi_1 \cdots \xi_r) \ := \ {\rm ord}_z (\xi_1) + \cdots + {\rm ord}_z(\xi_r).$$ A local differential operator on $M$ is {\em of weighted order $\leq k$} at $z$ if each of its monomial terms has weighted order at most $k$. 
\item[(iii)]
A local holomorphic function $\varphi \in \sO_{M,z}$ has {\em weighted vanishing order $\geq k+1$} at $z \in M$, if $P\varphi$ vanishes at $z$ for any local differential operator $P$ of weighted order $\leq k$. Denote by $\bm^k_z \subset \sO_z$ the ideal of local holomorphic functions with weighted vanishing order $\geq k$ at $z$ and set $ J^k_z := \sO_z / \bm^{k+1}_z.$ The vector bundle $J^k$ on $M$ whose fiber at $z$ is $J^k_z$ is called the {\em bundle of contact $k$-jets}. 
A local holomorphic function $\varphi \in \sO_{M,z}$ naturally determines a local section $j^k(\varphi) \in \sO(J^k)_z$, called the {\em contact $k$-jet} of $\varphi$ at $z$. 
There is a natural surjective homomorphism of vector bundles $ \pi^{k}_{k-1}: J^{k} \to J^{k-1}$ yielding an exact sequence of  vector bundles 
$$0 \to U_{k}(\fn)^* \longrightarrow J^{k} \stackrel{\pi^{k}_{k-1}}{\longrightarrow} J^{k-1} \to 0,$$ where $U_k(\fn)^*$ denotes the dual of the vector bundle $U_k(\fn)$.  \end{itemize} \end{definition}


\begin{notation}\label{n.coordi}
Using the coordinates $x_a, y_a, t$ on the Heisenberg group $\sH$ from Lemma \ref{l.coordi},  
for a multi-index $I = (i_1, \ldots, i_m, j_1, \ldots, j_m, p)$, define
\begin{eqnarray*} (xyt)^I &:= & (\prod_{a=1}^m x_a^{i_a} y_a^{j_a}) t^p, \\ |I| &:= & \sum_{s=1}^m(i_s + j_s) + 2p, \\ |I'| &:=& \sum_{s=1}^m(i_s + j_s) = |I| - 2p. \end{eqnarray*} For each $1 \leq a \leq m$, define
\begin{eqnarray*} 
I^{+a} & := &   (i_1, \ldots, i_{a}+1, \ldots,  i_m, j_1, \ldots, j_m, p), \\
I^{-a} & := &   (i_1, \ldots, i_{a}-1, \ldots,  i_m, j_1, \ldots, j_m, p+1), \\
I_{+a} & := &   (i_1, \ldots,  i_m, j_1, \ldots,  j_{a}+1, \ldots, j_m, p), \\
I_{-a} & := &   (i_1, \ldots,  i_m, j_1, \ldots,  j_{a}-1, \ldots, j_m, p+1), \end{eqnarray*} such that $|I^{+a}| = |I^{-a}| = |I_{+a}| = |I_{-a}| = |I| + 1$. Also define $$I^{\sharp} \ := \ (i_1, \ldots, i_m, j_1, \ldots, j_m, p+1) $$ such that $|I^{\sharp}| = |I| + 2.$  \end{notation}

We recall the weighted Taylor formula for functions on Heisenberg groups from \cite[Proposition 20.3.14, Corollary 20.3.15]{BLU} (see also  \cite[Section 4]{ACC}).

\begin{theorem}\label{t.Taylor}
In Notation \ref{n.coordi},  for any local holomorphic function $\varphi \in \sO_{\sH,0}$, the polynomial $$ \sum_{|I|=0}^k  \frac{\big( \{\prod_{a=1}^m X_a^{i_a} Y_a^{j_a}\}T^p \varphi \big)(0)}{|I'|! \ p !} (xyt)^I $$ with $I = (i_1, \ldots, i_m, j_1, \ldots, j_m, p)$ in the sum, belongs to $ \varphi + \bm_0^{k+1}$.  Here, we regard $X_a, Y_a, T$ as left-invariant vector fields on $\sH$ and take derivative of $\varphi$ with respect to these vector fields. Similarly, for any point $h \in \sH$ and $\varphi \in \sO_{\sH,h}$,  the polynomial $$ \sum_{|I|=0}^k  \frac{\big( \{\prod_{a=1}^m X_a^{i_a} Y_a^{j_a}\}T^p \varphi \big)(h)}{|I'|!  p !} h^*(xyt)^I $$ belongs to $\varphi + \bm_h^{k+1}$, where for any function $F$ on $\sH$, we denote by $h^*F$ the function whose value at $g \in \sH$ is $h^*F(g) = F(h^{-1} g)$. 
\end{theorem}

\begin{remark}\label{r.jet}
From the exact sequence in Definition \ref{d.contact} (iii) and Theorem \ref{t.Taylor}, we can identify $U_k(\fh)^*$ with the space of  polynomials of weighted degree $k$ in $(x_a, y_a, t)$. Thus the graded dual $U(\fh)^* = \oplus_{k=0}^{\infty} U_k(\fh)^*$ can be identified with the space of polynomial functions on the Heisenberg group $\sH$. \end{remark}

\begin{definition}\label{d.Spencer} 
For $0\le r\le k$, write $\pi^k_r:J^k\to J^r$ for the natural truncation.
By convention, we set $J^r=0$ and $\pi^k_r=0$ when $r<0$.
By Theorem \ref{t.Taylor}, the vector space $J^k_h$ for $h \in \sH$ is equal to the vector space of polynomials of the form   $$ \sum_{|I| = 0}^{ k} c_I \ h^*(xyt)^I $$ with coefficients $c_I \in \C$. Thus an element $\sigma \in \sO(J^k)_h$, namely,  a local section $\sigma$ of the bundle $J^k$ near $h$, can be written as $$ \sigma(g) =  \sum_{|I| = 0}^{k} \sigma_I(g) \ g^*(xyt)^I $$ for  $\sigma_I \in \sO_{\sH, h}$, which we write simply as $$ \sigma =  \sum_{|I| = 0}^{k} \sigma_I  \ g^*(xyt)^I.$$ For each $1 \leq a \leq m$, define the differential operator $D_{X_a}$ (resp. $D_{Y_a}$)  that sends $\sigma \in \sO(J^k)_h$ to an element of  $\sO(J^{k-1})_h$ as follows. 
\begin{eqnarray*} 
D_{X_a} \sigma & := & \sum_{|I|=0}^{k-1} (X_a \sigma_I - (i_a +1) \sigma_{I^{+a}} - \frac{p+1}{2} \sigma_{I_{-a}}) g^*(xyt)^I \\
D_{Y_a} \sigma & := & \sum_{|I|=0}^{k-1} (Y_a \sigma_I - (j_a +1) \sigma_{I_{+a}} + \frac{p+1}{2} \sigma_{I^{-a}}) g^*(xyt)^I.
\end{eqnarray*}
Also define the differential operator $D_{T}$  that sends  $\sigma \in \sO(J^k)_h$ to an element of  $\sO(J^{k-2})_h$ as follows. 
$$ D_T \sigma \ = \  \sum_{|I|=0}^{k-2} (T \sigma_I - (p+1) \sigma_{I^{\sharp}}) g^*(xyt)^I.$$
\end{definition}

\begin{remark} The operators $D_{X_a}, D_{Y_a}, D_T$ are adaptation of Spencer's differential operator $D$ in \cite[Section 1.3]{Sp} to contact manifolds. Spencer's differential operator $D$ is intrinsically defined, independent of coordinates or frames.  It may be possible to give such an intrinsic  formulation of our operators in Definition \ref{d.Spencer}, but we do not pursue it here because it is not needed in our discussion. \end{remark}

The following is an analogue  of \cite[Proposition 1.3.1]{Sp}, in the setting of a contact manifold.

\begin{theorem}\label{t.Spencer}
In the notation of Definition \ref{d.Spencer}, we have the following.
\begin{itemize} \item[(i)] For each $1 \leq b \leq m$, 
for any  $\varphi \in \sO_{\sH,h}$ and $\sigma \in \sO(J^k)_h$,
\begin{eqnarray*} D_{X_b} (\varphi \sigma) & = & (X_b \varphi) \pi^k_{k-1} \sigma + \varphi  \ D_{X_b} \sigma \\
D_{Y_b} (\varphi \sigma) & = & (Y_b \varphi) \pi^k_{k-1} \sigma + \varphi  \ D_{Y_b} \sigma \\
D_T (\varphi \sigma) & = & (T \varphi) \pi^k_{k-2} \sigma + \varphi  \ D_{T} \sigma. \end{eqnarray*}
\item[(ii)] 
An element $\sigma \in \sO(J^k)_h$ is of the form $j^k(\varphi)$ for some $\varphi \in \sO_h$ if and only if $$ D_{X_b} \sigma \ = \ D_{Y_b} \sigma \ = \ D_T \sigma \ =0$$ for all $1\leq b \leq m$. \end{itemize} \end{theorem} 

\begin{proof}
Part (i) follows by substituting the displayed formulas for $D_{X_b}$, $D_{Y_b}$ and $D_T$ and using the Leibniz rule for $X_b$, $Y_b$ and $T$.
Let us prove (ii). 
 Assume that $\sigma = j^k(\varphi)$. Then $\varphi$ must be equal to $\sigma_O \in \sO_{\sH, h}$ for the index $$O =(i_1 =0, \ldots,i_m= 0, j_1 = 0, \ldots, j_m=  0, p=0).$$ Moreover,  for each $I, 0 \leq |I| \leq k$, 
Theorem \ref{t.Taylor} implies  
\begin{equation}\label{eq.sigma} \sigma_I = \frac{ \{ \prod_{a=1}^m X_a^{i_a} Y_a^{j_a}\} T^p \varphi}{|I'|! \ p!}, \end{equation} namely, 
$$\sigma_I (g) = \frac{ (\{ \prod_{a=1}^m X_a^{i_a} Y_a^{j_a}\} T^p \varphi)(g)}{|I'|! \ p!} $$ for any $g \in \sH$ near $h$. 
  By Lemma \ref{l.Moyal}, we have for each $I$ with $0 \leq |I| \leq k-1$,   \begin{eqnarray*} X_b \sigma_I &=&
   \frac{ X_b \{ \prod_{a=1}^m X_a^{i_a} Y_a^{j_a}\} T^p \varphi}{|I'|! \ p!} \\ & = & \frac{(i_b + 1)}{(|I'| +1)} \frac{  \{ X_b^{i_b +1} Y_b^{j_b} \prod_{a\neq b} X_a^{i_a} Y_a^{j_a}\} T^p \varphi}{|I'|! \ p!} \\ & & 
   + \frac{|I'|}{2} \frac{  \{ X_b^{i_b} Y_b^{j_b -1} \prod_{a\neq b} X_a^{i_a} Y_a^{j_a}\} T^{p+1}  \varphi}{|I'|! \ p!} \\ &=& 
   (i_b +1) \sigma_{I^{+b}} + \frac{(p+1)}{2} \sigma_{I_{-b}}.\end{eqnarray*} This implies  $D_{X_b} \sigma =0$. By a similar argument, we obtain $D_{Y_b} \sigma =0.$ On the other hand, for each $I$ with $0 \leq |I| \leq k-2$, \begin{eqnarray*} T \sigma_I & = & \frac{ \{ \prod_{a=1}^m X_a^{i_a} Y_a^{j_a}\} T^{p+1} \varphi}{|I'|! \ p!} \\ &=& (p+1) \sigma_{I^{\sharp}}. \end{eqnarray*}
   This implies $D_T \sigma =0.$ 
   
   Conversely, assume that $D_{X_b} \sigma = D_{Y_b} \sigma = D_T \sigma = 0,$
   namely, \begin{equation}\label{eq.Xb} X_b \sigma_I = (i_b +1) \sigma_{I^{+b}} + \frac{p+1}{2} \sigma_{I_{-b}}, \end{equation}
   \begin{equation}\label{eq.Yb} Y_b \sigma_I = (j_b +1) \sigma_{I_{+b}} - \frac{p+1}{2} \sigma_{I^{-b}} \end{equation}
   for $I = (i_1, \ldots, i_m, j_1, \ldots, j_m, p)$ with $|I| \leq k-1$ and \begin{equation}\label{eq.T} T \sigma_I = (p+1) \sigma_{I^{\sharp}} \end{equation} for  $I = (i_1, \ldots, i_m, j_1, \ldots, j_m, p)$ with $|I| \leq k-2.$
   We claim that the collection of  functions  $\sigma_I, 0 \leq |I| \leq k$ satisfying  the differential equations (\ref{eq.Xb}), (\ref{eq.Yb}) and (\ref{eq.T}) is uniquely determined by $\varphi = \sigma_O$.  This is by induction on $|I'|$ and $|I|$. In fact, (\ref{eq.T}) says that $\sigma_J$ is determined by $\sigma_I$ with $|I'| = |J'|$ and $|I| < |J|$. 
   Then (\ref{eq.Xb}) and (\ref{eq.Yb}) say that $\sigma_J$ is determined if we know $\sigma_I$ for all $I$ with $|I'| < |J'|$.
    Since we have already checked that the collection of functions given by (\ref{eq.sigma}) satisfies (\ref{eq.Xb}), (\ref{eq.Yb}) and (\ref{eq.T}), the uniqueness implies that $$\sigma_I = \frac{ \{ \prod_{a=1}^m X_a^{i_a} Y_a^{j_a}\} T^p \varphi}{|I'|! \ p!}.$$ Thus $\sigma = j^k(\varphi).$   
\end{proof}

\begin{definition}\label{d.delta}
 For $h \in \sH$, the vector space $U_k(\fn_h)^*$ is a subspace of $J^k_h$ by the exact sequence in Definition \ref{d.contact} (iii). By Theorem \ref{t.Taylor}, it can be identified with the vector space of polynomials of the form   $$ \sum_{|I|= k} c_I \ h^*(xyt)^I $$ with coefficients $c_I \in \C$.  
  For each $1 \leq a \leq m$, define the  homomorphism  $\delta_{\sX_a,h}$ (resp. $\delta_{\sY_a,h}$)  from $U_k(\fn_h)^*$ to $U_{k-1}(\fn_h)^*$
 as follows. 
\begin{eqnarray*} 
\delta_{\sX_a,h} \big( \sum_{|I|= k} c_I \ h^*(xyt)^I \big) & := & (\frac{\p}{\p x_a} + \frac{y_a}{2} \frac{\p}{\p t}) \big( \sum_{|I|= k} c_I \ h^*(xyt)^I \big)
 \\
\delta_{\sY_a,h} \big( \sum_{|I|= k} c_I \ h^*(xyt)^I \big) & := & (\frac{\p}{\p y_a} - \frac{x_a}{2} \frac{\p}{\p t}) \big( \sum_{|I|= k} c_I \ h^*(xyt)^I \big),
\end{eqnarray*}
where on the righthand side, the vector fields act on the polynomial functions on $\sH$ as derivative. 
Also define the homomorphism $\delta_{T,h}: U_k(\fn_h)^* \to U_{k-2}(\fn_h)^*$ by
\begin{eqnarray*} 
\delta_{T,h} \big( \sum_{|I|= k} c_I \ h^*(xyt)^I \big) & := &  \frac{\p}{\p t} \big( \sum_{|I|= k} c_I \ h^*(xyt)^I \big).
\end{eqnarray*}
They define vector bundle homomorphisms $$\delta_{\sX_a}, \delta_{\sY_a}: U_k(\fn)^* \to U_{k-1}(\fn)^*, \ \delta_T: U_k(\fn)^* \to U_{k-2}(\fn)^*$$ and corresponding sheaf homomorphisms, which we denote by the same symbols,
$$\delta_{\sX_a}, \delta_{\sY_a}: \sO(U_k(\fn)^*) \to \sO(U_{k-1}(\fn)^*), \ \delta_T: \sO(U_k(\fn)^*) \to \sO(U_{k-2}(\fn)^*).$$
\end{definition}

\begin{lemma}\label{l.Seashi}
In Definition \ref{d.delta}, we have the following commutative diagrams.
$$ \begin{array}{ccc}
\sO(U_k(\fn)^*) & \stackrel{- \delta_{\sX_a}}{\to} & \sO(U_{k-1}(\fn)^*) \\ \downarrow & & \downarrow \\
\sO(J^k)  & \stackrel{D_{X_a}}{\to} & \sO(J^{k-1}), \end{array}  
\begin{array}{ccc}
\sO(U_k(\fn)^*) & \stackrel{- \delta_{\sY_a}}{\to} & \sO(U_{k-1}(\fn)^*) \\ \downarrow & & \downarrow \\
\sO(J^k)  & \stackrel{D_{Y_a}}{\to} & \sO(J^{k-1}), \end{array} $$ and $$
\begin{array}{ccc}
\sO(U_k(\fn)^*) & \stackrel{- \delta_{T}}{\to} & \sO(U_{k-2}(\fn)^*) \\ \downarrow & & \downarrow \\
\sO(J^k)  & \stackrel{D_{T}}{\to} & \sO(J^{k-2}), \end{array} $$
where all the vertical arrows are from Definition \ref{d.contact} (iii). 
\end{lemma}

\begin{proof} As mentioned in Definition \ref{d.Spencer}, a local section of $J^k$ is of the form $$\sum_{|I|= 0}^k \sigma_I \ g^*(xyt)^I.$$ It is a local section of the vector subbundle $U_k(\fn)^* \subset J^k$ when $\sigma_I =0$ for $|I|< k$. Then by the definition of $D_{X_a}$ in Definition \ref{d.Spencer}, we have 
\begin{eqnarray*} D_{X_a} \left( \sum_{|I|= k} \sigma_I \ g^*(xyt)^I \right) &=& - \sum_{|I| = k-1}
((i_a +1) \sigma_{I^{+a}} + \frac{(p+1)}{2} \sigma_{I_{-a}})g^*(xyt)^I 
\\ & = & - \delta_{\sX_a} \left( \sum_{|I|= k} \sigma_I \ g^*(xyt)^I \right), \end{eqnarray*} which verifies the first commutative diagram. The other two can be proved by similar arguments. 
\end{proof}

\section{Contact fundamental forms}\label{s.cFF}

The following is a natural noncommutative version of the definition of the system of fundamental forms of a submanifold in projective space in  \cite[Definition 2.4]{FH}.

\begin{definition}\label{d.CF}
Let $(M, H \subset TM)$ be a contact manifold and let $\sL$ be a line bundle on $M$. Twisting the exact sequence in Definition \ref{d.contact} (iii), we have the exact sequence of vector bundles
\begin{equation}\label{eq.jet} 0 \to U_{k}(\fn)^*\otimes \sL \longrightarrow J^{k} \otimes \sL  \longrightarrow J^{k-1}\otimes \sL \to 0. \end{equation}
Suppose we have a fixed embedding $M \subset \BP^N$ as a complex submanifold of projective space. Let $\sL$ be the restriction of the hyperplane line bundle on $\BP^N$ to $M$. The embedding is given by a subspace $W \subset H^0(M, \sL)$.   For a point $z \in M$, we have a natural inclusion $W \subset \sO(\sL)_z$, determining a subspace  $j^k_z(W) \subset  J^k_z \otimes \sL_z$ for all $k \geq 0$. Define $$W^k_z := W \cap \mathbf{m}_z^k\cdot \sO(\sL)_z.$$  Then $j^{k}_z (W^{k+1}_z)=0$ and (\ref{eq.jet}) determine a  homomorphism $${\rm CF}^k_z : W_z^k/W_z^{k+1} \to U_k(\fn_z)^*\otimes \sL_z,$$ which we call the $k$-th {\em contact fundamental form } of $M \subset \BP^N$ at $z$. Let $S^k_z \subset U_k(\fn_z)^*$ be the image of ${\rm CF}^k_z$, modulo a local trivialization of $\sL$ at $z$. The graded subspace $S_z := \oplus_k S^k_z$ of $U(\fn_z)^*$ is called the {\em system of contact fundamental forms}. It is easy to see that $S_z$ is independent of the choice of a local trivialization of $\sL$ near $z$.   \end{definition}

An alternative formulation is the following.

\begin{lemma}\label{l.cap}
In Definition \ref{d.CF}, for each $z \in M$, after choosing a local trivialization of $\sL$ near $z$, write $\sW^k_z:= j^k_z(W)  \subset  J^k_z$  and regard $U_k(\fn_z)^*$ as a subspace of $J^k_z$ by the exact sequence in Definition \ref{d.contact} (iii).  Then $$S^k_z = \sW^k_z \cap U_k(\fn_z)^*.$$ \end{lemma}

\begin{proof}

Fix a local trivialization of $\sL$ near $z$. Let us define a homomorphism $\beta: \sW^k_z\cap U_k(\fn_z)^* \to S^k_z$.  For 
$u\in\sW^k_z\cap U_k(\fn_z)^*$, we can  choose
$w\in W$ with $j^k_z(w)=u$ by  the definition of $\sW^k_z$.  Since $u$ lies in
$U_k(\fn_z)^*=\ker(J^k_z\to J^{k-1}_z)$, the $(k-1)$-jet of $w$ at $z$
vanishes, hence $w\in W^k_z$.  Define
$$
\beta(u):={\rm CF}^k_z(w\ {\rm mod}\ W^{k+1}_z) \ \in S^k_z.
$$
This is well-defined: if $w'$ is another representative with
$j^k_z(w')=u$, then $j^k_z(w-w')=0$, so $w-w'\in W^{k+1}_z$ and the class
modulo $W^{k+1}_z$ is unchanged.

The map $\beta$ is injective.  If $\beta(u)=0$, choose $w\in W^k_z$ with
$j^k_z(w)=u$.  The vanishing of $\beta(u)$ means that the class of $w$ in
$W^k_z/W^{k+1}_z$ maps to zero under the contact fundamental form, hence
$j^k_z(w)=0$ in the kernel $U_k(\fn_z)^*$; therefore $u=0$.

The map $\beta$ is surjective by construction. In fact, every element of $S^k_z$ is
${\rm CF}^k_z(w\ {\rm mod}\ W^{k+1}_z)$ for some $w\in W^k_z$; its
$k$-jet lies in $\sW^k_z$ and, because its lower weighted jet vanishes, lies
in $U_k(\fn_z)^*$.  Thus it is the image under $\beta$ of an element of
$\sW^k_z\cap U_k(\fn_z)^*$.

Thus we have showed that $S^k_z$ and $\sW^k_z \cap U_k(\fn_z)^*$ are isomorphic by the homomorphism $\beta$. But it is easy to see that $\beta$ identifies the two subspaces of $U_k(\fn_z)^*$, proving the lemma.  
\end{proof}

Now we are ready to prove Theorem \ref{t.cCartan}.

\begin{proof}[Proof of Theorem \ref{t.cCartan}] 
For any $z \in M$, since there are elements of $W$ nonvanishing at $z$, we have $S_z^0=\C$.  Since the linear system $W$ defines an
embedding, the first weighted jets of affine coordinate functions coming from $W$ give $S_z^1=\fn_{z,1}^*$.   
Consider the homomorphism
$$
\rho_{2,z}:U_2(\fn_z)^*\longrightarrow \fn_{z,2}^*
$$ induced by $\fn_{z,2}\hookrightarrow U_2(\fn_z)$.
 Since the linear system $W$ defines an immersion, the
  differential at $z$ of some affine coordinate function coming from $W$ which vanishes at $z$  determines a nonzero element of 
  $T_z^*M$ annihilating $H_z$.  Thus it has a nonzero image under
$\rho_{2,z}$.  Since $\fn_{z,2}^*$ is one-dimensional, this proves
$\rho_{2,z}(S_z^2)=\fn_{z,2}^*$.

It remains to verify that $S_z$ for a general $z \in M$ is preserved under the operation of the right-invariant vector fields on the Heisenberg group of the Lie algebra $\fn_z$. 
For each $z \in M$, the descending sequence of
kernels of $j_z^k:W\to J_z^k\otimes\sL_z$ for $k >0$ stabilizes at zero: a
section of $\sL$ whose germ vanishes to arbitrary higher order is zero.  Since $W$ is
finite-dimensional, this happens for some finite $k_0$.  Over a Zariski-open subset of $M$,
the jet-image spaces
$$
\sW_z^k =j_z^k(W)\subset J_z^k\otimes\sL_z
$$
and their intersections with $U_k(\fn_z)^*\otimes\sL_z$ have constant rank
for $0\le k\le k_0$, while $S_z^k=0$ for $k>k_0$.  We work on this Zariski-open subset.

After a local trivialization of $\sL$, the subspaces $\sW_z^k$ form vector
subbundles $\sW^k\subset J^k$ on this open subset, and the intersections in
Lemma \ref{l.cap} form the corresponding subbundles of contact fundamental
forms.
By Darboux theorem for contact structures,  each point $z \in M^o$ admits a neighborhood $z \in O_z \subset M^o$ with a contactomorphic biholomorphism $\alpha: O_z \cong O_0$ of a neighborhood $0 \in O_0 \subset \sH$ in the Heisenberg group $\sH$ of the same dimension as $M$. Let us identify $O_z$ with $O_0$. 

We claim that differential operators $D_{X_a}$ and $ D_{Y_a}$  send $\sO(\sW^k) \subset \sO(J^k)$ to $\sO(\sW^{k-1}) \subset \sO(J^{k-1})$ and the differential operator $D_T$ sends $\sO(\sW^k) \subset \sO(J^k)$ to $\sO(\sW^{k-2}) \subset \sO(J^{k-2})$. In fact, an element of $\sO(\sW^k)$ can be expressed as $\sum_i \varphi_i w_i$ where $w_i$'s are sections of $J^k$ given by the contact $k$-jet of  an element of $W$ and $\varphi_i$ is a local holomorphic function. Then $D_{X_a} w_i =0$ by Theorem \ref{t.Spencer} (ii) and  $$D_{X_a}(\varphi_i w_i) = X_a(\varphi_i) \pi^k_{k-1} w_i + \varphi_i D_{X_a} w_i = X_a(\varphi_i) \pi^k_{k-1} w_i \ \in \ \sO(\sW^{k-1})$$
by Theorem \ref{t.Spencer} (i). This proves the claim for $D_{X_a}$. The same argument proves the claim for $D_{Y_a}$ and $D_T$.
 
By the above claim, the commutative diagrams in Lemma \ref{l.Seashi} and Lemma \ref{l.cap},  the system of contact fundamental forms must be preserved under $\delta_{\sX_a}, \delta_{\sY_a}$ and $\delta_{T}$. This proves the theorem.   
\end{proof}

\section{Heisenberg-symmetric varieties}\label{s.Hsymmetric}

\begin{definition}\label{d.Hsymmetric}
Let $S \subset U(\fh)^*$ be a contact symbol system on the graded Heisenberg algebra $\fh$.
Regarding $U(\fh)^*$ as the space of  polynomial functions on the Heisenberg group $\sH$, the subspace $S \subset U(\fh)^*$ is a linear system of  regular functions on the affine variety $\sH$. This gives rise to a morphism $\Phi_S: \sH \to \BP S^*$. Let $Z^S \subset \BP S^*$ be the closure of the image of $\Phi_S$. We call it the {\em Heisenberg-symmetric variety} determined by the contact symbol system $S \subset U(\fh)^*$. 
\end{definition}

Theorem \ref{t.cFH} is a consequence of the following two propositions.  
\begin{proposition}\label{p.Hsymmetric}
In Definition \ref{d.Hsymmetric}, the following holds.
\begin{itemize}
\item[(i)] The morphism $\Phi_S$ is an embedding and  left translations of $\sH$ can be extended to an action of $\sH$ on $\BP S^*$ that preserves the subvariety $Z^S \subset \BP S^*$.
\item[(ii)]
For each $z \in \sH$, the contact fundamental form $S_z$ of the contact manifold $(\sH, H \subset T\sH)$ embedded in $\BP S^*$ 
is isomorphic to $S \subset U(\fh)^*$ by a graded Lie algebra isomorphism $\fn_z \cong \fh$. 
\item[(iii)] For each $z \in \sH$, there exists a $\C^*$-action on $\BP S^*$ which preserves $\sH \subset Z^S$ and its contact structure, with an isolated fixed point at $z$, such that the induced $\C^*$-action on $H_z$ is by scalar multiplications.
    \end{itemize} \end{proposition}
    
    \begin{proof}
     Since $S^1 = \fh_1^*$ and $\rho_2 (S^2) = \fh_2^*$, the morphism $\Phi_S$ is immersive. 
    The linear system $S$ on $\sH$ is preserved under the operation of right-invariant vector fields.   A right-invariant vector field on $\sH$ generates a one-parameter subgroup of left translations on $\sH$ (see \cite[Chapter I, Proposition 4.1]{KN}). Thus the linear system $S$ is preserved under left translations on $\sH$.   It follows that $\sH$ acts on $\BP S^*$ preserving $Z^S$.
    Since $\Phi_S$ is algebraic, the isotropy subgroup of the $\sH$-action at a general point of $Z^S$ must be a finite subgroup.
    But $\sH$ has no nontrivial finite subgroup. If follows that $\Phi_S$ is an embedding. 
    This proves (i).
    

By the transitive $\sH$-action from (i), it suffices to check (ii) and (iii) at the identity element $0\in\sH$.

For (ii),   since $S$ is graded by weighted degree,
$$
S\cap\bm_0^k=\bigoplus_{r\geq k}S^r,\qquad
\frac{S\cap\bm_0^k}{S\cap\bm_0^{k+1}}\simeq S^k.
$$
Under the identification of Remark \ref{r.jet}, this quotient is precisely the $k$-th contact fundamental form.  Hence the full system at the identity element is $S$.

For (iii), for $\lambda\in\C^*$ define the Heisenberg dilation
$$
\delta_\lambda(v+w)=\lambda v+\lambda^2w,\qquad
v\in\fh_1,\quad w\in\fh_2.
$$
It is a group automorphism of $\sH$ preserving the standard contact distribution.  Since $S$ is graded, pullback by $\delta_\lambda$ preserves $S$, so the evaluation embedding is equivariant for a projective linear $\C^*$-action on $\BP S^*$.  In the affine chart $U_1$, this action has weight $1$ on the contact directions and weight $2$ on the quotient direction.  Thus the identity element is an isolated fixed point and the action on its contact hyperplane is by scalar multiplications .  
\end{proof}

%

\begin{proposition}\label{p.unique}
Let $Z \subset \BP^N$ be a nondegenerate projective variety with a Zariski-open subset $M \subset Z$ equipped with a contact structure $H \subset TM$.
Assume, for a fixed contact symbol system $S\subset U(\fh)^*$, that $M \subset \BP^N$ satisfies the conditions (i) and (ii) in Theorem \ref{t.cFH}. Then for each point $z \in M$, there exists a projective linear  isomorphism $\BP^N \cong \BP S^*$, which sends a Zariski-open neighborhood $z \in O_z \subset M$ to $\sH \subset Z^S$ biregularly, preserving the contact structures. \end{proposition}

\begin{proof}
Fix $z\in M$ and write $\BP^N=\BP W^*$.  Let
$
F^kW:=W_z^k
$
be the weighted vanishing filtration of the ambient linear system $W$ at $z$.
By Definition \ref{d.CF} and Lemma \ref{l.cap}, condition (i) identifies
the associated graded space with $S_z\simeq S$; in particular $\dim W=\dim S$.
Let $\lambda:\C^*\hookrightarrow{\rm PGL}(W)$ be the subgroup in condition
(ii).  Choose a linear lift of $\lambda$ to $W$, normalized so that it acts
trivially on $F^0W/F^1W$.  Replacing $\lambda(t)$ by $\lambda(t^{-1})$ if
necessary, assume that the weight on $H_z$ is $a>0$.  The nondegenerate Levi
bracket $\wedge^2H_z\to T_zM/H_z$ is equivariant, so the quotient
$T_zM/H_z$ has weight $2a$.  It follows that the induced action on
$
F^kW/F^{k+1}W\simeq S_z^k
$
has weight $-ka$. Since all
projective weights are multiples of $a$ and $\lambda$ is an effective
one-parameter subgroup of ${\rm PGL}(W)$, we have $a=1$.

Since $\C^*$ is reductive, the filtration splits equivariantly:
$$
W=\bigoplus_{k\ge0}W_k,\qquad
F^kW=\bigoplus_{j\ge k}W_j,\qquad
W_k\simeq S_z^k.
$$
The space $W_0$ is one-dimensional because $S_z^0=\C$.    Passing to the dual,
we obtain
$$
W^*=\C v_0\oplus\bigoplus_{k>0}W_k^*,
$$
where $z=[v_0]$, the line $\C v_0$ has weight $0$, and all other weights are
positive.

Choose $\ell_0\in W_0$ with $v_0(\ell_0)=1$.  The attracting set of $z$ (in the sense of \cite[Section 2.4]{CG}) in
$\BP(W^*)$ is the affine chart
$$
A_z
=\{\,[\alpha]\in\BP(W^*)\mid \alpha(\ell_0)\ne0\,\}
=\left\{[v_0+w]\ \middle|\ w\in\bigoplus_{k>0}W_k^*\right\}.
$$

Let $O_z:=Z\cap A_z$, which is a $\C^*$-stable affine open neighborhood of $z$ in $Z$.  
If there exists any $x\in O_z\cap(Z\setminus M)$, then the closed
$\C^*$-invariant set $Z\setminus M$ contains $\lambda(t)x$ for all
$t\ne0$, and hence also its limit $z$, which is a contradiction. 
This shows that $O_z\subset M$, thus $O_z$ is a
smooth affine neighborhood of $z$ in $M$.  By the Bialynicki-Birula decomposition (see \cite[Theorem 2.4.3]{CG}), we have a
$\C^*$-equivariant isomorphism $O_z\simeq\A^{2m+1}$
with coordinates $z_1, \ldots z_{2m}, s$ of weights $1,\ldots,1,2$.

Now let us show that the contact structure on $O_z$ can be put into Heisenberg normal form.    Since $\operatorname{Pic}(O_z)=0$, the $\lambda$-linearized line bundle  $\operatorname{Ann}(H|_{O_z})\subset\Omega^1_{O_z}$ admits  a nowhere-vanishing algebraic section $\theta$. As $\lambda$ preserves the contact structure,  there exists  an invertible regular function $u_t$ on $O_z$ such that 
$\lambda(t)^*\theta=u_t\theta$.
 Such a function is constant, and the constants $u_t$ form a character; hence $\theta$ is semi-invariant.  Its weight is $2$, because $\theta_z$ spans the annihilator of $H_z$, dual to the weight-two quotient direction.  It follows that
$$
\theta=c\,{\rm d} s +\sum_{\alpha,\beta}a_{\alpha\beta}z_\alpha\, {\rm d} z_\beta,
\qquad c\ne0,
$$
where the $z_\alpha$ are the weight-one coordinates.  After rescaling $s$, we can set $c=-2$.  Replacing $s$ by $s+q(z)$ for a suitable homogeneous quadratic polynomial removes the symmetric part of $(a_{\alpha\beta})$.  The remaining skew form is nondegenerate because $\theta$ is contact, and a linear symplectic change of variables gives
$$
\theta= \sum_{a=1}^m(p_a\,{\rm d} q_a-q_a\, {\rm d} p_a) -2 {\rm d}s.
$$
Thus $(O_z,H|_{O_z})$ is contactomorphic to the standard Heisenberg group, and condition (i) identifies its symbol algebra with the fixed algebra $\fh$.

Finally, we identify the restricted projective linear system.  Choose a $\lambda$-eigenvector $\ell_0\in W$ with $\ell_0(z)\ne0$.  The description of $A_z$ shows that $\ell_0$ is nowhere zero on $O_z$.  Set
$$
V:=\left\{\left.\frac{\ell}{\ell_0}\right|_{O_z}\ \middle|\ \ell\in W\right\}
\subset\C[O_z].
$$
Since $O_z$ contains a nonempty Zariski-open subset of the nondegenerate variety $Z$, the restriction map $W\to V$ is injective, hence an isomorphism.  The space $V$ is $\C^*$-stable, so it is a direct sum of weighted-homogeneous subspaces.  Each weighted-homogeneous polynomial equals its weighted initial form; consequently the natural map $V\to\gr_zV$ is an isomorphism of graded vector spaces.  By Definition \ref{d.CF}, $\gr_zV=S_z$, and condition (i) identifies this system with $S$.  Hence $W\simeq S$ as graded linear systems on the Heisenberg group.  The induced projective linear isomorphism
$\BP W^*\simeq\BP S^*$
sends the embedding of $O_z$ to $\Phi_S(\sH)$ and preserves the contact structures.
\end{proof} 

We now examine some examples of  Heisenberg-symmetric varieties. The first result says that 
for a rank 2 contact symbol system, its associated Heisenberg-symmetric variety is always a cone, hence it is
smooth only if it is the ambient projective space.

\begin{proposition}\label{p.rank-two-cone}
Let $S\subset U(\fh)^*$ be a contact symbol system of rank $2$.  Then
$Z^S\subset\BP S^*$ is a projective cone.  It is smooth if and only if
$Z^S=\BP S^*$.
\end{proposition}

\begin{proof}
Choose $F_0\in S^2$ with $\rho_2(F_0)\ne0$ and write
$$
S^2=\C F_0\oplus K,\qquad K:=\ker(\rho_2|_{S^2}).
$$
We first exhibit the cone structure.  Choose a linear coordinate $t$ on
$\fh_2$.  In exponential coordinates
$(\xi,t)$ on the Heisenberg group, every element of $U_2(\fh)^*$ has the form
$c\,t+Q(\xi)$ with $Q\in\Sym^2\fh_1^*$.  Thus the elements of
$K=\ker(\rho_2|_{S^2})$ are precisely the quadratic functions in
$S^2$ with no $t$-term.  Write
$$
F_0(\xi,t)=c\,t+Q_0(\xi),\qquad c\ne0.
$$
Let $\mathbf 1$ be the constant function in $S^0$ and put
$$
U_{\mathbf 1}:=\{[\lambda]\in\BP S^*\mid \lambda(\mathbf 1)\ne0\}.
$$
On $U_{\mathbf 1}$ we normalize $\lambda(\mathbf 1)=1$.  With respect to the
decomposition
$$
S=\C\mathbf 1\oplus S^1\oplus K\oplus \C F_0,
$$
the point $\Phi_S(\xi,t)$ has affine coordinates
$[1:\xi:K(\xi):F_0(\xi,t)]$.  For fixed $\xi$, the last coordinate
$F_0(\xi,t)$ is arbitrary as $t$ varies.  Hence
$$
Z^S\cap U_{\mathbf 1}
=\{[1:\xi:K(\xi):a]\mid \xi\in\fh_1,\ a\in\C\}\subset
\BP(S^0\oplus S^1\oplus K\oplus\C F_0)^*.
$$
 Taking its projective
closure gives the projective cone with vertex $v=[0:0:0:1]$ over the closure of
$$
\fh_1\longrightarrow \BP(S^0\oplus S^1\oplus K)^*,\qquad
\xi\longmapsto [1:\xi:K(\xi)].
$$

Being a cone, it is smooth if and only if the base is linear, namely if and only if $K=0$, which is equivalent to $Z^S = \BP S^*.$
\end{proof}

We  now  give a  class of  examples of smooth Heisenberg-symmetric varieties: some Segre varieties. 
  For an $m$-dimensional vector space $V$, 
let
$$
\fh(V) := V\oplus V^*\oplus \C z $$ be the Heisenberg algebra with non-trivial Lie brackets given by
 $$  [{\mathbf x},{\mathbf y}]={\mathbf y}({\mathbf x}) z \  \mbox{ for } \ {\mathbf x} \in V, {\mathbf y} \in V^*.
$$
Let $\sH(V)$ be the corresponding Heisenberg group.  In exponential
coordinates $({\mathbf x},{\mathbf y},t),$ set
$$
\tau=t-\frac12{\mathbf y}({\mathbf x}).
$$
For ${\mathbf x}_0\in V$ and ${\mathbf y}_0\in V^*$, the right-invariant vector fields are
$$
\sX_{{\mathbf x}_0}=\partial_{{\mathbf x}_0},\qquad
\sY_{{\mathbf y}_0}=\partial_{{\mathbf y}_0}
-{\mathbf y}_0({\mathbf x})\partial_\tau,\qquad
T=\partial_\tau .
$$
In fact, by Lemma \ref{l.coordi} (i), the group law in the coordinates
$({\mathbf x},{\mathbf y},\tau)$ is
$$
({\mathbf x},{\mathbf y},\tau)\cdot({\mathbf x}',{\mathbf y}',\tau')
=({\mathbf x}+{\mathbf x}',{\mathbf y}+{\mathbf y}',
\tau+\tau'-{\mathbf y}({\mathbf x}')).
$$
This follows by substituting $t=\tau+\frac12{\mathbf y}({\mathbf x})$ in the
Campbell--Hausdorff formula.  The right-invariant vector field with value
${\mathbf x}_0$ at the identity is obtained by differentiating
$$
s\longmapsto (s{\mathbf x}_0,0,0)\cdot({\mathbf x},{\mathbf y},\tau),
$$
which gives $\partial_{{\mathbf x}_0}$.  Similarly, differentiating
$$
s\longmapsto (0,s{\mathbf y}_0,0)\cdot({\mathbf x},{\mathbf y},\tau)
$$
gives $\partial_{{\mathbf y}_0}-{\mathbf y}_0({\mathbf x})\partial_\tau$,
and differentiating
$s\longmapsto (0,0,s)\cdot({\mathbf x},{\mathbf y},\tau)$ gives
$\partial_\tau$.

Fix $0 < r < \dim V$ and let 
$$ {\mathbf b} := (
V^*=B_1\oplus\cdots\oplus B_r) $$ be a decomposition of $V^*$ into $r$ positive-dimensional factors. 
Writing ${\mathbf y}={\mathbf y}_1+\cdots+{\mathbf y}_r$ with ${\mathbf y}_j\in B_j$,
define
$$
\Psi_{\mathbf b}:\sH(V)\longrightarrow
\BP(\C\oplus V\oplus \C)\times\prod_{j=1}^r\BP(\C\oplus B_j)
$$
by
$$
({\mathbf x},{\mathbf y}, t)\longmapsto
\big([1:{\mathbf x}:\tau],[1:{\mathbf y}_1],\ldots,[1:{\mathbf y}_r]\big).
$$
Let $S_{\mathbf b}$ be the $\Psi_{\mathbf b}$-pullback of the Segre linear system:
$$
S_{\mathbf b}:=(\C\oplus V^*\oplus \C\tau)
\prod_{j=1}^r(\C\oplus B_j^*)\subset \C[\sH(V)]=U(\fh(V))^*.
$$

\begin{proposition}\label{p.product-examples}
The graded subspace $S_{\mathbf b}\subset U(\fh(V))^*$ is a contact symbol
system with ${\rm rank}(S_{\mathbf b})=r+2$.  Its associated Heisenberg-symmetric variety is the Segre embedding
$$
\BP^{m+1}\times\BP^{b_1}\times\cdots\times\BP^{b_r}
\subset \BP(S_{\mathbf b}^*), \qquad b_j = \dim B_j.
$$
\end{proposition}

\begin{proof}
The vector field $\sX_{{\mathbf x}_0}$ differentiates only the $V^*$-coordinates in the
first factor, while $T$ sends $\tau$ to $1$.  The vector field
$\sY_{{\mathbf y}_0}$ differentiates the $B_j^*$-coordinates in the factors
$\C\oplus B_j^*$, and its term $-{\mathbf y}_0({\mathbf x})\partial_\tau$ sends
$\C\tau$ into $V^*$.  Hence the right-invariant vector fields preserve
$S_{\mathbf b}$.  Moreover
$$
S_{\mathbf b}^0=\C,\qquad
S_{\mathbf b}^1=V^*\oplus B_1^*\oplus\cdots\oplus B_r^*
\cong V^*\oplus V=\fh(V)_1^*
$$
This gives the required degree-one component.  Since $T(\tau)=1$, the restriction map
$S_{\mathbf b}^2\to\fh(V)_2^*$ is surjective.  Thus $S_{\mathbf b}$ is a
contact symbol system.

The map $\Psi_{\mathbf b}$ identifies $\sH(V)$ with the product of the
standard affine charts in
$$
\BP(\C\oplus V\oplus\C)\times\prod_{j=1}^r\BP(\C\oplus B_j).
$$
The morphism $\Phi_{S_{\mathbf b}}$ in Definition \ref{d.Hsymmetric} is the composition of $\Psi_{\mathbf b}$
with the Segre embedding.  Taking closures gives
$$
Z^{S_{\mathbf b}}=
\BP(\C\oplus V\oplus\C)\times\prod_{j=1}^r\BP(\C\oplus B_j)
=\BP^{m+1}\times\BP^{b_1}\times\cdots\times\BP^{b_r}.
$$
The first factor contributes weights $0,1,2$, and each remaining factor
contributes weights $0,1$.  Hence the highest nonzero weighted degree is
$r+2$, realized by multiplying $\tau$ with one nonzero linear coordinate from
each $B_j^*$.  This proves ${\rm rank}(S_{\mathbf b})=r+2$.
\end{proof}

\section{Contact fundamental forms of adjoint varieties}\label{s.adjoint}

For the discussions in this section, we need to recall some results on graded
simple Lie algebras.  
The following two theorems are well-known (see \cite[Proposition 8.4.5]{Spr} and 
\cite[Sections 3 and 4.2]{Ya}).


\begin{theorem}\label{t.graded}
  Any  grading of a complex  simple Lie algebra $\fg$  is  of the form $$ \fg \ = \ \fg_{-d} \oplus \fg_{-d+1} \oplus \cdots \oplus \fg_{-1} \oplus \fg_0 \oplus \fg_1 \oplus \cdots \oplus \fg_{d -1} \oplus \fg_d$$ for some positive integer $d$ such that 
 writing $\fg_-:=\oplus_{i<0}\fg_i$ and $\fg_+:=\oplus_{i>0}\fg_i$, \begin{itemize} \item[(i)] the subalgebra $\fp_-:= \fg_0 \oplus \fg_-$ (resp. $\fp_+ := \fg_0 \oplus \fg_+$)  is a parabolic subalgebra of $\fg$, namely, the homogeneous space $G/P_-$ (resp. $G/P_+$) is projective for the Lie groups  $P_- \subset G$ (resp. $P_+ \subset G$) of the Lie algebras $\fp_- \subset \fg$ (resp. $\fp_+ \subset \fg$); and \item[(ii)] $\fp_-$ and $\fp_+$ are isomorphic Lie algebras.  \end{itemize} \end{theorem}

\begin{theorem}\label{t.adjoint}
  For each simple Lie algebra $\fg$, there is a unique (up to isomorphism) graded Lie algebra structure of the form 
  $$\fg \ = \ \fg_{-2} \oplus \fg_{-1} \oplus \fg_0 \oplus \fg_1 \oplus  \fg_2$$ such that $\fg_-$ and $\fg_+$ are graded Heisenberg algebras. The  corresponding homogeneous space $G/P_+$ (in the notation of Theorem \ref{t.graded}), called the {\em adjoint variety} of $\fg$, has a natural $G$-invariant contact structure given by $\fg_{-1}$. Conversely, if a $G$-homogeneous projective variety has a $G$-invariant contact structure, then it is the adjoint variety of $\fg$. \end{theorem}


Recall the following from \cite{Be}.  

\begin{proposition}\label{p.Beauville}
For a simple Lie group $G$ and its algebra $\fg$, let  $\sO \subset \fg^*$ be a  coadjoint orbit of a nonzero nilpotent element. Then \begin{itemize} \item[(i)] its projectivization $M:= \BP \sO$ has a $G$-invariant contact structure $H \subset TM$ such that the line bundle $L = TM/H$ agrees with the restriction $\sL$ of the hyperplane line bundle on $\BP \fg^*$ and \item[(ii)] under the isomorphism $\fc \simeq H^0(M, L)$ from Theorem \ref{t.Ko}, the contactomorphic vector field associated with $A\in\fg$, 
is sent to the restriction to $M$ of the ambient linear function
$\xi\longmapsto \xi(A)$
on $\fg^*$, determining an inclusion corresponding to the linear system
$$\fg = H^0(\BP \fg^*, \sO(1)) \subset H^0(M, L) \cong \fc.$$  Hence a linear automorphism of $\BP \fg^*$ preserving $M \subset \BP \fg^*$ and the contact structure $H \subset TM$ preserves this copy of $\fg \subset \fc$. \end{itemize}
\end{proposition}

\begin{proof}
(i) is from \cite[Section (2.4)]{Be} and (ii) is from \cite[Section (1.5)]{Be}. \end{proof}

%
%

  We determine when the closure of $\BP \sO$ in $\BP \fg^*$ is a Heisenberg-symmetric variety as follows.

\begin{theorem}\label{t.nilpotent}
In Proposition \ref{p.Beauville}, let $Z \subset \BP \fg^*$ be the closure of $M= \BP \sO$.   Then $Z \subset \BP \fg^*$ is a  Heisenberg-symmetric variety if and only if $M =Z,$  namely, it is the adjoint variety of $\fg$ in Theorem \ref{t.adjoint}. \end{theorem}

We use the following lemma.

\begin{lemma}\label{l.Lie}
Let $(M,H\subset TM)$ be a contact manifold, and let
$\fc$ be the Lie algebra of all contactomorphic  vector fields.
\begin{itemize}
\item[(i)] If $\psi:M\to M$ is a contact automorphism, then
$\psi_*:H^0(M,TM)\to H^0(M,TM)$ induces a Lie algebra automorphism of
$\fc$.
\item[(ii)] Under the isomorphism $\fc \simeq H^0(M,TM/H)$ in Theorem \ref{t.Ko}, the
subalgebra of contactomorphic vector fields vanishing at $z\in M$ corresponds to
$H^0(M,(TM/H)\otimes\mathbf m_z^2),$
where $\mathbf m_z^2$ is the weighted ideal from Definition \ref{d.CF}.
\end{itemize}
\end{lemma}

\begin{proof}
Part (i) is immediate.  For (ii), the assertion is local at $z$.  In Darboux
coordinates, equivalently in the Heisenberg model of Proposition \ref{p.Heis},
the vector field $\overrightarrow{F}$ corresponding to a local holomorphic function $F$ vanishes at the origin exactly
when the weighted
vanishing order of $F$ is at least $2$.
\end{proof}

\begin{proof}[Proof of Theorem \ref{t.nilpotent}]
Suppose first that $M=Z$ is the adjoint variety from Theorem \ref{t.adjoint}.  For its contact
grading
$$
\fg=\fg_{-2}\oplus\fg_{-1}\oplus\fg_0\oplus\fg_1\oplus\fg_2,
$$
the nilpotent group $\exp(\fg_-)$ acts simply transitively on a big
cell, where $\fg_-=\fg_{-2}\oplus\fg_{-1}$ is a Heisenberg algebra and the
contact distribution is the left-invariant distribution determined by
$\fg_{-1}$.  The grading element gives the Heisenberg dilation, with weights
$1$ on $\fg_{-1}$ and $2$ on $\fg_{-2}$.  At the base point this gives the
required $\C^*$-action; the fixed line is isolated because the extremal space
$\fg_2$ is one-dimensional.  Conjugating by $G$ gives such an action at
every point, while $G$-homogeneity identifies all contact fundamental systems.
Thus $Z$ is Heisenberg-symmetric.

Conversely, suppose that $Z\subset\BP\fg^*$ is Heisenberg-symmetric.  Recall
first that $Z$ is nondegenerate in $\BP\fg^*$. By the definition of a Heisenberg-symmetric embedding, there are a contact
symbol system
$$
S=\bigoplus_{k\geq0}S^k\subset U(\fh)^*
$$
and a projective linear identification of $Z\subset\BP\fg^*$ with
$Z^S\subset\BP S^*$, under which the Heisenberg open orbit $\sH$ is a contact
open subset of $M$.  We use this identification and take its identity
$z\in\sH$.  The embedding of $\sH$ is the map $\Phi_S$ of Definition
\ref{d.Hsymmetric}; consequently its restricted ambient linear system is
precisely the graded system $S$.  On the other hand, Proposition \ref{p.Beauville} (ii) identifies the ambient linear system with $\fg$
through Theorem \ref{t.Ko}.  We therefore have a vector-space isomorphism
$
\kappa:S\longrightarrow\fg.
$

For $F,G\in S$, define the bracket $\{F, G \}$ by
$
\overrightarrow{\{F,G\}}=[\overrightarrow F,\overrightarrow G].
$
The right-hand side corresponds to an element of $\fg$ and hence $\{F, G \}$ is another
element of $S$.  If $F\in S^p$ and $G\in S^q$, Proposition \ref{p.Heis} shows
that $\overrightarrow F$ and $\overrightarrow G$ have weighted degrees $p-2$
and $q-2$, respectively.  It follows that
$$
\{F,G\}\in S^{p+q-2},
\qquad
\kappa(\{F,G\})=[\kappa(F),\kappa(G)].
$$
For $k\geq0$, define $\fg_{k-2}:=\kappa(S^k).$
Then $[\fg_i,\fg_j]\subset\fg_{i+j}$, so this is a Lie algebra grading of
$\fg$. Since $\fg_{-3} =0$, it must be of the form 
\begin{equation}\label{eq.contact}
\fg=\fg_{-2}\oplus\fg_{-1}\oplus\fg_0\oplus\fg_1\oplus\fg_2 \end{equation} from Theorem \ref{t.graded}. 

We claim that $[\fg_{-1}, \fg_{-1}] \neq 0$. Otherwise $[\fg_1, \fg_1] =0$ by Theorem \ref{t.graded} and  $\fg_{-1} \oplus \fg_0 \oplus \fg_1$ is an ideal of $\fg$, a contradiction to the simplicity of $\fg$. 
 Since $\dim\fg_{-2}=1$ because $S^0=\C$, we see by \cite[Lemma 2.1]{Ya} that (\ref{eq.contact}) agrees with the graded Lie algebra in Theorem \ref{t.adjoint}.

%

Finally, Lemma \ref{l.Lie} (ii) says that the isotropy subalgebra at $z$ is
$$
\fg_z=\fg_0\oplus\fg_1\oplus\fg_2=\fp_+.
$$
By Theorem \ref{t.adjoint},  the stabilizer of $z$ is $P_+$ and
 $M\simeq G/P_+$ is projective.  As $M$ is dense in $Z$, it follows
that $M=Z$, and it is the adjoint variety of $\fg$.
\end{proof}

Note that the argument in the above proof of Theorem \ref{t.nilpotent} gives the following information on the contact fundamental forms of adjoint varieties.

\begin{lemma}\label{l.adjoint}
Let $\fg=\fg_{-2}\oplus\fg_{-1}\oplus\fg_0\oplus\fg_1\oplus\fg_2$
be the contact grading of Theorem \ref{t.adjoint}, and let $Z\subset\BP\fg^*$
be the adjoint variety with base point $z$ whose isotropy subalgebra is
$\fg_0\oplus\fg_1\oplus\fg_2$.  Then the isomorphism in Theorem \ref{t.Ko} 
induces natural graded vector-space isomorphisms
$S_z^k\simeq \fg_{k-2}$ for all $k \geq 0$.
\end{lemma}

\begin{remark}
A more explicit description of the contact fundamental form of the adjoint
variety can be read from the expression of the rational map $\phi$ in
\cite[page 140]{LM02}.
\end{remark}

To prove Theorem \ref{t.cLM},  we need to recall the notion of the universal prolongation from
\cite[Section 2.3]{Ya}.

\begin{definition}\label{d.prolong}
Let $\fn_-=\fn_{-1}\oplus\cdots\oplus\fn_{-d}$
be a negatively graded nilpotent Lie algebra, and let
$\fn_0\subset{\rm Der}_{\rm gr}(\fn_-)$ be a Lie subalgebra such that
$\fn_-\oplus\fn_0$ is a graded Lie algebra.  For $i>0$, after the spaces
$\fn_j$ for $j<i$ and their brackets with $\fn_-$ have been constructed,
define $\fn_i$ to be the space of degree-$i$ linear maps
$$
\phi:\fn_-\longrightarrow\bigoplus_{j<i}\fn_j,\qquad
\phi(\fn_{-k})\subset\fn_{i-k},
$$
satisfying
$$
\phi([u,v])=[\phi(u),v]+[u,\phi(v)]\qquad (u,v\in\fn_-).
$$
Here brackets involving $\fn_0$ and the previously constructed positive
components are the recursively defined evaluation brackets on $\fn_-$.  The
graded Lie algebra
$$
\operatorname{prol}(\fn_-,\fn_0):=\bigoplus_{i\ge-d}\fn_i
$$
is called the {\em universal prolongation} of $(\fn_-,\fn_0)$.
\end{definition}

\begin{lemma}\label{l.universal}
Let $\mathfrak q$ be a Lie algebra containing vector subspaces $V_i$, $i\ge-d$.
Assume that
\begin{itemize}
\item[(i)] $V_-\oplus V_0$, where $V_-:=\oplus_{i=-d}^{-1}V_i$, is a graded
Lie subalgebra isomorphic to $\fn_-\oplus\fn_0$;
\item[(ii)] $[V_i,V_{-k}]\subset V_{i-k}$ for $i\ge1$ and $1\le k\le d$,
where $V_r=0$ for $r<-d$; and
\item[(iii)] $\{X\in V_i\mid[X,V_-]=0\}=0$ for each $i\ge1$.
\end{itemize} Notice that the
spaces $V_i$ are not required to be closed under brackets with one another.
Then there are injective linear maps
$$
\jmath_i : \ V_i\hookrightarrow\operatorname{prol}(\fn_-,\fn_0)_i = \fn_i \qquad(i\ge1).
$$
\end{lemma}

\begin{proof}
Identify $V_-\oplus V_0$ with $\fn_-\oplus\fn_0$, and let $\jmath_r: V_r \to \fn_i$  denote
this identification for $-d\le r\le0$.  We construct $\jmath_i: V_i \to \fn_i$ for $i>0$ inductively.  Assume that the
maps $\jmath_r$ have been constructed for $r<i$.  For $X\in V_i$, define
$$
\phi_X(u):=\jmath_{i-k}([X,u])\qquad(u\in V_{-k}).
$$
By condition (ii),  this is a degree-$i$ map to $\oplus_{r<i} \fn_r$.  The Jacobi identity in $\mathfrak q$, together with the induction
hypothesis, gives
$$
\phi_X([u,v])=[\phi_X(u),v]+[u,\phi_X(v)],
$$
so $\phi_X$ lies in $\fn_i$.  Set
$\jmath_i(X):=\phi_X$.  If $\phi_X=0$, then $[X,V_-]=0$, hence $X=0$ by
(iii).  This proves injectivity and completes the induction.  
\end{proof}

The following is a special case of Yamaguchi's prolongation theorem \cite[Theorem 5.2]{Ya}.

\begin{theorem}\label{t.Yamaguchi}
Let $\fg=\fg_{-2}\oplus\fg_{-1}\oplus\fg_0\oplus\fg_1\oplus\fg_2$
be the grading of a complex simple Lie algebra from Theorem \ref{t.adjoint}.  Unless $\fg$ is the symplectic Lie algebra, the
universal prolongation of $(\fg_-,\fg_0)$ is isomorphic to $\fg$ as a graded
Lie algebra.  
\end{theorem}

Now we are ready to prove Theorem \ref{t.cLM}.


\begin{proof}[Proof of Theorem \ref{t.cLM}]
Set $\widetilde S_z^k:=S_z^k$ for $k\le2$.  For $k\ge3$, define
$$
\widetilde S_z^k:=
\left\{
F\in U_k(\fn_z)^* \ \middle|\
\begin{array}{ll}
\sX_v(F)\in\widetilde S_z^{k-1} & \mbox{for all }v\in\fn_{z,1},\\
\sX_v(F)\in\widetilde S_z^{k-2} & \mbox{for all }v\in\fn_{z,2}
\end{array}
\right\}.
$$
The adjoint variety is homogeneous, so the conclusion of Theorem
\ref{t.cCartan}, initially stated at a general point, holds at the chosen
point $z$.  Hence $S_z^k\subset\widetilde S_z^k$ for all $k$.  It suffices to
prove
\begin{equation}\label{eq.equ}
\dim S_z^k=\dim\widetilde S_z^k\qquad(k\ge3).
\end{equation}

We can identify a Zariski-open neighborhood of $z$  with the Heisenberg
group of the symbol algebra $\fn_z$.  Regarding $U(\fn_z)^*$ as polynomial
functions on this neighborhood, Proposition \ref{p.Heis} gives the injective homomorphism 
 into the Lie algebra $\fc$ of contactomorphic vector fields on $\mathcal H$:
$$
\kappa:U(\fn_z)^*\longrightarrow \fc,\qquad
F\longmapsto \overrightarrow F .
$$
Lemma \ref{l.bracket} gives
\begin{equation}\label{eq.kappa}
\kappa(\sX_v(F))=[\sX_v,\kappa(F)]\qquad(v\in\fn_z).
\end{equation}
For $i\ge-2$, put
$$
\fa_i:=\kappa(\widetilde S_z^{i+2}).
$$
Since $\widetilde S_z^k=S_z^k$ for $k\le2$, Lemma \ref{l.adjoint} gives an isomorphism 
$$
\fa_{-2}\oplus\fa_{-1}\oplus\fa_0
\simeq
\fg_{-2}\oplus\fg_{-1}\oplus\fg_0
$$
as graded Lie algebras.

The defining conditions of $\widetilde S_z^{i+2}$ and \eqref{eq.kappa} give
$$
[\fa_i,\fa_{-1}]\subset\fa_{i-1},\qquad
[\fa_i,\fa_{-2}]\subset\fa_{i-2}\qquad(i\ge1).
$$
For a function $F$ on $\sH$, if   $X=\kappa(F)\in\fa_i, i \geq 1$ commutes with
$\fa_{-2}\oplus\fa_{-1}$, then all derivatives of $F$ with respect to the right-invariant vector fields vanish.
Thus $F$ is constant on the Heisenberg group.  Since $F$ is weighted
homogeneous of degree $i+2\ge3$, it is zero, and $X=0$.  Lemma
\ref{l.universal}, applied to $d =1, V_i = \fa_i$, gives
$$
\fa_i\hookrightarrow \operatorname{prol}(\fg_-,\fg_0)_i .
$$

Assume now that $\fg$ is not of symplectic type.  By Theorem
\ref{t.Yamaguchi},
$
\operatorname{prol}(\fg_-,\fg_0)\simeq\fg .
$
Hence, for $i\ge1$,
$$
\dim\widetilde S_z^{i+2}=\dim\fa_i\le\dim\fg_i .
$$
On the other hand, $S_z^{i+2}\subset\widetilde S_z^{i+2}$ and Lemma
\ref{l.adjoint} gives $\dim S_z^{i+2}=\dim\fg_i$.  Therefore
$$
\dim\fg_i=\dim S_z^{i+2}\le
\dim\widetilde S_z^{i+2}\le\dim\fg_i,
$$
so equality holds throughout.  This proves \eqref{eq.equ}, hence Theorem
\ref{t.cLM}.
\end{proof}

\begin{remark}\label{r.symplectic-exception}
For symplectic $\fg$, the equality in Theorem \ref{t.cLM} fails in weighted
degree $3$.  The adjoint variety is the second Veronese embedding of
projective space.  At a point $z$, the contact fundamental form is 
$$
S_z^0=\C,\quad S_z^1=E^*,\quad
S_z^2=\Sym^2E^*\oplus\C\tau,\quad
S_z^3=E^*\tau,\quad S_z^4=\C\tau^2,
$$ and $S_z^k=0$ for $k\ge5$, for some vector space $E$ and an element $\tau$ of weight 2.  If $0\ne F\in\Sym^3E^*\subset U_3(\fn_z)^*$, then  $F\in\widetilde S_z^3$ (in the notation of the proof of Theorem \ref{t.cLM}), but $F\notin S_z^3$.
\end{remark}

 \quad \\

Baohua Fu (bhfu@math.ac.cn), State Key Laboratory of Mathematical Sciences, Morningside Center of Mathematics, Academy of Mathematics and Systems Science, Chinese Academy of Sciences, Beijing 100190, China;   and School of Mathematical Sciences, University of Chinese Academy of Sciences, Beijing, China \\

Jun-Muk Hwang (jmhwang@ibs.re.kr),  Center for Complex Geometry, Institute for Basic Science (IBS), Daejeon 34126, Republic of Korea

 \end{document}